\documentclass[sn-mathphys-num]{sn-jnl}

\usepackage{graphicx}%
\usepackage{multirow}%
\usepackage{amsmath,amssymb,amsfonts}%
\usepackage{amsthm}%
\usepackage{mathrsfs}%
\usepackage[title]{appendix}%
\usepackage{xcolor}%
\usepackage{textcomp}%
\usepackage{manyfoot}%
\usepackage{booktabs}%
\usepackage{algorithm}%
\usepackage{algorithmicx}%
\usepackage{algpseudocode}%
\usepackage{listings}%

\newtheorem{theorem}{Theorem}%
\newtheorem{proposition}[theorem]{Proposition}%

\newtheorem{remark}{Remark}%

\newtheorem{definition}{Definition}%

\newtheorem{lemma}{Lemma}
\newtheorem{corollary}[theorem]{Corollary}

\theoremstyle{remark}

\begin{document}

\title[A notion of fractional slice monogenic functions with respect to a pair of real valued functions]{A notion of fractional slice monogenic functions with respect to a pair of real valued functions}

\author[1]{\fnm{Jos\'e Oscar} \sur{Gonz\'alez Cervantes}}\email{jogc200678@gmail.com}

\author*[2]{\fnm{Juan} \sur{Bory Reyes}}\email{juanboryreyes@yahoo.com}

\affil*[1]{\orgdiv{Departamento de Matem\'aticas}, \orgname{ESFM-Instituto Polit\'ecnico Nacional}, \orgaddress{\street{Unidad Profesional ``Adolfo L\'opez Mateos", Col. Lindavista,}, \city{Ciudad M\'exico}, \postcode{07338}, \country{M\'exico}}}

\affil[2]{\orgdiv{SEPI}, \orgname{ESIME-Zacatenco-Instituto Polit\'ecnico Nacional}, \orgaddress{\street{Unidad Profesional ``Adolfo L\'opez Mateos", Col. Lindavista,}, \city{Ciudad M\'exico}, \postcode{07338}, \country{M\'exico}}}

\abstract{This work presents  the basic elements and results of a Clifford algebra valued fractional slice monogenic functions theory  defined from the  null-solutions of a suitably fractional Cauchy-Riemann operator in the Riemann-Liouville and Caputo sense with respect to a pair of real valued functions on certain domains of Euclidean spaces.}



\maketitle

\section{Introduction} 
Slice monogenic functions are defined on domains of Euclidean spaces having values in a Clifford algebra. The literature is very rich of results and the studies on the topic are ongoing, see for instance \cite{ACDS, CGS, CDR, CKPS, CSS1, CSS2, CSS3, CS1, CS2, YQ, XS}.

Fractional  integrals and derivatives are mathematical tools widely used in several branches of science and engineering. Fractional calculus, which addresses the derivatives and integrals with arbitrary real or complex order, is nowadays an extensively developed topic which the reader can approach in the classical references \cite{KST, JAA, OS, O, P, MR, SKM}.

Clifford analysis is a function theory studying null solutions of the Cauchy-Riemann or Dirac systems, called monogenic functions, which are defined on domains of Euclidean spaces and with values in a Clifford algebra. Standard reference books are \cite{BDS, CSSS, GM, GHS}.

In recent years, various extension of the Fractional calculus into Clifford analysis (in particular, into quaternionic analysis) has attracted more and more attention in the literature. Some examples of effective establishments can be found in \cite{CTOP, DM, GB1, GB2, GB3, GBS, KV, PBBB, V}. 

Inspired by \cite{GBS}, where the notion of fractional slice regular functions of a quaternionic variable defined as null-solutions of a fractional Cauchy-Riemann operator is introduced, we further generalize these ideas into twofold directions: one is a notion of fractional slice monogenic function theory and the other is the study of null-solutions of a fractional Cauchy-Riemann operator with respect to two real valued functions, which can be exponential or lineal functions. In any case we expand the theory presented in \cite{GBS}.

\section{Preliminaries}
In this section, we collect some preliminary results on Fractional calculus and on Slice monogenic functions theory.
\subsection{Fractional integral and derivative in Riemann-Liouville and Caputo sense with respect to another function}

Suppose that $-\infty <a  < b< \infty$ and  $\alpha\in \mathbb C$ with $\Re \alpha> 0$.  Let  $g\in C^1([a,b], \mathbb R)$ be such that  $g'(t)>0$ for all $t\in [a,b]$. 
Let us recall that the Riemann-Liouville integrals of order $\alpha$ of   $f  \in L^1([a, b], \mathbb R)$, with respect to $g$, on the left and on the right, are defined  by 
$$({\bf I}_{a^+}^{\alpha, g} f)(x) := \frac{1}{\Gamma(\alpha)} \int_a^x \frac{f(\tau)}{(g(x)- g(\tau))^{1-\alpha}} g'(\tau) d\tau, \quad \textrm{with}  \quad x > a$$
and
$$({\bf I}_{b^-}^{\alpha,g} f)(x) := \frac{1}{\Gamma(\alpha)} \int_x^b \frac{f(\tau)}{(g(\tau)- g(x))^{1-\alpha}} g'(\tau)d\tau, \quad \textrm{with}  \quad x < b,$$
respectively. 

The fractional derivatives in the Riemann-Liouville sense, on the left and on the right, with respect to $g$, are defined by 
\begin{align}\label{FracDer} 
	({}_{RL}D _{a^+}^{\alpha, g } f)(x):= \frac{1}{ g'(x)} \frac{d}{  dx} \left[ ({\bf I}_{a^+}^{1-\alpha, g} f)(x)\right]
\end{align}
and
\begin{align} \label{FracDer1}
	({}_{RL}D _{b^-}^{\alpha, g} f)(x):= \frac{1}{ g'(x)} \frac{d}{   dx}\left[({\bf I}_{b^-}^{1-\alpha, g}f)(x)\right] 
\end{align}
respectively. It is worth noting that the derivatives in  \eqref{FracDer} and \eqref{FracDer1} exist for $f\in AC^1([a, b], \mathbb R)$. Fractional Riemann-Liouville integral and derivative are linear operators.

Fundamental theorem for Riemann-Liouville fractional calculus \cite{JAA, JARH} shows that 
\begin{align}\label{FundTheorem}
	({}_{RL}D_{a^+}^{\alpha, g} {\bf I}_{a^+}^{\alpha, g}f)(x)=f(x) \quad  \textrm{and} \quad ({}_{RL}D _{b^-}^{\alpha, g}  {\bf I}_{b^-}^{\alpha, g} f)(x) = f(x).
\end{align}

From \cite[Definition 2.7]{JARH} we see that given $f  \in C^1([a, b], \mathbb R)$, the fractional derivatives in the Caputo  sense, on the left and on the right, with respect to $g$,  are defined by 
\begin{align}\label{FracDerCaputo} 
	({}_C D _{a^+}^{\alpha, g } f)(x):= & ({\bf I}_{a^+}^{1-\alpha, g}     \frac{f' }{ g' }  )     (x) , \nonumber\\ 
	({}_C D _{b^-}^{\alpha, g} f)(x):= &  ({\bf I}_{b^-}^{1-\alpha, g} \frac{f'}{g'})(x), 
\end{align}
respectively. 

Finally, if $f \in C^1([a, b], \mathbb R)$ we have 
\begin{align}\label{R_LandC1}
	({}_C D _{a^+}^{\alpha, g } ({\bf I}_{a^+}^{1-\alpha, g} f))(x):= & ({\bf I}_{a^+}^{1-\alpha, g} \frac{1 }{g'} \frac{d}{dx} ({\bf I}_{a^+}^{1-\alpha, g} f))(x) = ({\bf I}_{a^+}^{1-\alpha, g} {}_{RL} D _{a^+}^{\alpha, g } f) (x),
\end{align}
for all $x\in [a,b]$.

Similarly, 
\begin{align}\label{R_LandC2} 
	({}_C D _{b^-}^{\alpha, g} ({\bf I}_{a^+}^{1-\alpha, g} f))(x):=  & ({\bf I}_{b^-}^{1-\alpha, g} \frac{1}{g'} {}_{RL}D _{b^-}^{\alpha, g} f )(x),
\end{align}
for all $x\in [a,b]$. 

\subsection{Slice monogenic functions}
We first give some basic knowledge in relation to Clifford algebra \cite{BDS, DSS}. Let $\{e_1,\dots , e_n\}$  imaginary units satisfying $e_ie_j +e_je_i = -2\delta_{ij}$.  Clifford algebra $\mathbb R_n $ is formed by $\sum_{A} e_Ax_A$, where $A = \{i_1,\dots, i_r\}$, $ i_{\ell}\in \{1, 2, . . . , n\}$ and  $i_1 < \cdots  < i_r$ is a multi-index, $e_A = e_{i_1}e_{i_2} \cdots e_{i_r}$ and $e_{\emptyset}  = 1$,  $x_A\in\mathbb  R$. An element $(x_1,\dots , x_n) \in \mathbb R_n$ is identified with a 1-vector in $\mathbb R_n$  through the map  
$$(x_1, \dots , x_n) \mapsto  \underline{x} = x_1e_1 +\cdot + x_ne_n.$$ 

In addition, $(x_0, x_1,\dots , x_n) \in\mathbb  R^{n+1}$  will be identified with $x = x_0 +\underline{ x }= x_0 + \sum_{j=1}^n x_je_j$, which is called a paravector. The norm of $x \in \mathbb R^{n+1}$  is 
$|x|^2 = x^2_0 + x^2_1 +\cdots + x^2_n$. The real part $x_0$ of $x$ is  written as $Re[x]$. Given $x\in \mathbb R^{n+1}$ and $r>0$ denote 
$\mathbb D(x,r):=\{y\in\mathbb R^{n+1} \ \mid \ |x-y|<r\}$.

The sphere of unit 1-vectors in $\mathbb R^n$ will be given by $\mathbb S = \{\underline{x} = e_1x_1 + \cdots  + e_nx_n \ \mid \  x^2_1
+ \cdots + x^2_n = 1\}$ and given $\mathcal I \in \mathbb S$ we define 
$\mathbb C_{\mathcal I}:=\{ u+\mathcal  I v  \ \mid \  u, v\in \mathbb  R\},$ which is a 2-dimensional real subspace of $\mathbb R^{n+1}$ isomorphic to $\mathbb C$ as fields. 

In addition, given a paravector $x = x_0 +\underline{ x} \in  \mathbb R^{n+1}$ denote 
\begin{align}\label{unitvector}  
	\mathcal I_x =\left\{
	\begin{array}{l }\frac{\underline{x}}{|\underline{x}|} \quad  if \ \  \underline{x} \neq  0,\\
		\textrm{any element of $\mathbb S$ otherwise.}
	\end{array}
	\right.
\end{align}
Therefore, $\mathbb R^{n+1} =\bigcup_{ \mathcal I \in\mathbb S}  \mathbb C_{\mathcal I}$.  
\begin{definition} (Slice monogenic functions) \cite{CSS1, CSS3}. Suppose that $U \subset \mathbb R^{n+1}$ be an open set and let
	$f: U \to \mathbb R_n$ be a real differentiable function. Let $\mathcal I \in \mathbb  S$  and let $f_\mathcal I$ be the restriction
	of $f$ to $U_{\mathcal I}:=\mathbb C_\mathcal I \cap U$. Therefore, $f$ is called a (left) slice monogenic function,
	or s-monogenic function, if 
	\begin{align*}
		\overline{\partial}_{{\mathcal  I}}  f _{\mathcal I} (u + \mathcal I v) : =  \frac{1}{2}
		\left( \frac{\partial}{\partial u} + \mathcal I \frac{\partial }{\partial  v}\right)
		f_{\mathcal I} (u + \mathcal I v) = 0, \quad \textrm{on }U_{\mathcal I},
	\end{align*}
	for every $\mathcal I \in \mathbb S$. 
\end{definition}
In what follows, $\mathcal {SM}(U)$ stands for the  set of slice monogenic functions defined on $U$.
\begin{definition}(Slice domains). A domain  $U \subset \mathbb R^{n+1}$  is called slice domain (s-domain for short) if $ U \cap R$  is nonempty and if $ U_{ \mathcal I}$ is a domain in $\mathbb C_{\mathcal I}$ for all $\mathcal I \in \mathbb S$.
\end{definition}

\begin{definition} (Axially symmetric domains). A domain $ U \subset \mathbb  R^{n+1}$ is called 
	axially symmetric if, for all $x = x_0+\mathcal I_x 
	|\underline{x}|		 \in  U$ we have that 
	$\{x_0+ \mathcal I  |\underline{x}| \}\subset U$ for all $\mathcal I\in \mathbb S^2$.
\end{definition}

\begin{theorem} \label{RepresentationFormula}(Representation formula) \cite{CSS1, CSS3}. Let $U \subset \mathbb R^{n+1}$ be an axially symmetric s-domain. If $f\in \mathcal {SM}(U)$ and  $x = u + \mathcal I_x v \in  U$, where $u,v\in \mathbb R$, then  
	\begin{align*}
		f(x) = \frac{1}{2} \left(1 - \mathcal I_x \mathcal I \right) f(u + \mathcal Iv) + \frac{1}{2} \left(1 + \mathcal I_x \mathcal I \right) 	
		f(u- \mathcal  Iv), \quad \forall \mathcal I \in  \mathbb S.
	\end{align*}
\end{theorem}

\begin{lemma}\label{SplittingLemma} (Splitting Lemma). Let $U \subset \mathbb  R^{n+1}$  be an open set and $f\in \mathcal{SM}( U)$.   For every $\mathcal I = \mathcal I_1 \in \mathbb  S$, let $\mathcal I_2,\dots  ,\mathcal  I_n$  be
	a completion to a basis of $\mathbb R_n$ satisfying the defining relations $\mathcal I_r \mathcal I_s + \mathcal I_s \mathcal I_r = - 2\delta_{rs}$.
	Then there exist $2^{n-1}$ many holomorphic functions $F_A : U_{\mathcal I}  \to \mathbb C_{\mathcal I}$ such that
	\begin{align*}
		f_{\mathcal  I}(z)=\sum^{n-1}_{|A|=0}F_A(z) {\mathcal I}_A, \end{align*}
	where  ${\mathcal I}_A={\mathcal I}_{i_1}\cdots{\mathcal I}_{i_s}$ and $A=\{i_1,\dots , i_s\}$ is a subset of $\{2,\dots, n\}$, with $i_1<\dots< i_s$, and $I_{\emptyset}= 1$.
\end{lemma}

\begin{remark}
	If $U \subset \mathbb R^{n+1}$, we will use the symbol $\textrm{Hol}(U_{\mathcal I})$ to denote the complex linear space of holomorphic functions from $U_{\mathcal I}$ to $\mathbb C_{\mathcal I}$.
\end{remark}
Let $x=x_0 + \underline{x} = u + \mathcal I_x v  \in \mathbb R^{n+1}$ be a nonzero paravector, where $u = x_0$, $v = |\underline{x}|$ and $\mathcal I_x = {v}^{-1}\underline{x}$, with $v \neq 0$ and for $v = 0$, see \eqref{unitvector}. Moreover, if
$\mathcal I_x = (\zeta_1,\dots,\zeta_n)$ then $x_k = v\zeta_k$  for $k = 1,\dots, n$, see \cite{CSS1, CSS3} for more details.

From now on, $U\subset \mathbb R^{n+1}$ denotes an axially symmetric s-domain.
\begin{proposition} \label{properties}  
	Suppose that $f\in \mathcal {SM}(U)$ and $\mathcal I \in \mathbb S$. If $\overline{\mathbb D(y,r)}\subset U$, then
	\begin{align*}    
		f(z) = \frac{1}{2\pi }\int_{\partial \mathbb D(y,r)_{\mathcal I} }  (w- z)^{-1} d_w\sigma(\mathcal I)   f (w), \quad \forall z\in \mathbb D(y,r)_{\mathcal I}, 
	\end{align*}
	where $d_w\sigma(\mathcal I) = - (d_w\sigma) \mathcal I $ and ${\partial \mathbb D(y,r)_{\mathcal I} }$ is the boundary of $\mathbb D(y,r)_{\mathcal I}$ in $U_{\mathcal I}$. 
	In addition, 
	\begin{align*} 
		\int_{\Gamma} d_w\sigma(\mathcal I)   f (w) =0, 
	\end{align*}
	for all $\mathcal I\in \mathbb S$ and for any closed, homotopic to a point and piecewise $C^1$ curve $\Gamma\subset U_{\mathcal I}$. 
	
	On the other hand, if the set  $Z_f \cap U_{\mathcal I} = \{z \in U_{\mathcal I} \ \mid \ f(z) = 0\}$ has an accumulation point in $U_{\mathcal I}$, then $f \equiv 0$  on $U$. 
	
	Finally, if $\ell\in C(U, \mathbb R_n)$ satisfies   
	\begin{align*} \int_{\Gamma} d_w\sigma(\mathcal I)  \ell (w) =0,
	\end{align*}
	for any closed, homotopic to a point and piecewise $C^1$ curve $\Gamma\subset U_{\mathcal I}$ and for all $\mathcal I \in \mathbb S$, then $\ell\in \mathcal{SM} (U)$.
\end{proposition}
\begin{proof}
	The previous facts are direct consequence of the Cauchy formula, Cauchy theorem, identity principle and Morera's theorem respectively in the complex plane $\mathbb C_{\mathcal I}$ using also Representation Formula and Splitting Lemma, see \cite{CKPS, CSS1, CSS2, CSS3, CS1, CS2}.
\end{proof}

It should be pointed out that previous formulas may be extended on all axially symmetric s-domain $U$ but we will not develop this point here.

The following definition is inspired in the work \cite{G}, which presents a generalization of the  quaternionic slice regular functions. 

\begin{definition}\label{GeneralizedSM}
	Suppose that $\lambda\in \mathbb R$. If $f\in C^1( U, \mathbb R_n)$ satisfies  
	\begin{align*}
		\overline{\partial}_{{\mathcal  I}}  f _{\mathcal I} (u + \mathcal I v) +\lambda f _{\mathcal I}  = 0 , \quad \textrm{on  } U_{\mathcal I}
	\end{align*}
	for every $\mathcal I \in \mathbb S$, then $f$ is called $\lambda$-slice monogenic functions defined on $U$ and we write $\mathcal {SM}_\lambda (U)$ for the set of all $\lambda$-slice monogenic functions on $U$.
\end{definition}
Clearly $\mathcal {SM}_0 (U) = \mathcal {SM} (U)$. In the remainder of this section we require $\lambda\in \mathbb R$. 

A direct calculation shows
\begin{proposition} 
	If $f\in C^1( U, \mathbb R_n)$, then 
	\begin{align*}
		\overline{\partial}_{{\mathcal  I}} \left[ 
		e^{u\lambda}    f _{\mathcal I} (u + \mathcal I v) \right] = e^{u\lambda}  \left[
		\overline{\partial}_{{\mathcal  I}}  f _{\mathcal I} (u + \mathcal I v) +\lambda f _{\mathcal I}  \right] , \quad \forall  u + \mathcal I v \in U_{\mathcal I},
	\end{align*}
	for all $\mathcal I \in \mathbb S$.
\end{proposition}
As a consequence one gets
\begin{corollary}\label{corSMLambda}
	Let $f\in C ^1( U,\mathbb R_n)$. The mapping  $u+\mathcal I v \mapsto  e^{u\lambda} f _{\mathcal I} (u + \mathcal I v) $, for all $u+\mathcal I v\in U$, where $u,v\in \mathbb R$ and $\mathcal I\in \mathbb S$, belongs to $\mathcal {SM}(U)$ if and only if $f\in \mathcal {SM}_\lambda (U)$.
\end{corollary} 
From now on we make the assumption: $f \in  \mathcal {SM}_\lambda (U)$. 
\begin{proposition}\label{cor-Rep-Split-lambda}
	
	\noindent
	
	I. (Representation Theorem) Let $x = u + \mathcal I_x v \in  U$, where $u,v\in \mathbb R$ and $\mathcal I_x \in \mathbb S$, then  
	\begin{align}\label{represtheoremlambda}
		f(x) = \frac{1}{2} \left( 1 - \mathcal I_x \mathcal I \right) f(u + \mathcal Iv) + \frac{1}{2} \left(1 + \mathcal I_x \mathcal I \right) 	
		f(u- \mathcal  Iv), \quad \forall \mathcal I \in \mathbb S.
	\end{align}
	
	\noindent
	
	II. (Splitting Lemma) For every $\mathcal I = \mathcal I_1 \in \mathbb  S$, let $\mathcal I_2,\dots  ,\mathcal  I_n$  be
	a completion to a basis of $\mathbb R_n$ satisfying   $\mathcal I_r \mathcal I_s + \mathcal I_s \mathcal I_r = - 2\delta_{rs}$.
	Then there exist $2^{n-1}$ many holomorphic functions $F_A : U_{\mathcal I}  \to \mathbb C_{\mathcal I}$ such that
	\begin{align}\label{serieslambda}
		f_{\mathcal  I}(z)=\sum^{n-1}_{|A|=0}e^{-\lambda \Re z  }F_A(z) {\mathcal I}_A, 
	\end{align}
	where ${\mathcal I}_A={\mathcal I}_{i_1}\cdots{\mathcal I}_{i_s}$ and $A=\{i_1,\dots , i_s\}$ is a subset of $\{2,\dots, n\}$, with $i_1<\dots< i_s$, or $I_{\emptyset}= 1$.
\end{proposition}

\begin{proof} It follows from Corollary \ref{corSMLambda}, Theorem \ref{RepresentationFormula} and Lemma \ref{SplittingLemma}.
\end{proof}

\begin{proposition} \label{propertieslambda}  Given $\mathcal I \in \mathbb S$. If  $\overline{\mathbb D(y,r)}\subset U$, then
	\begin{align*}    
		f(z) = \frac{1}{2\pi }\int_{\partial \mathbb D(y,r)_{\mathcal I} }e^{\lambda (\Re w-\Re z)}(w- z)^{-1} d_w\sigma(\mathcal I) f (w), \quad \forall z\in \mathbb D(y,r)_{\mathcal I}. 
	\end{align*}
	In addition, 
	\begin{align*}    
		\int_{\Gamma} & e^{\lambda \Re w}d_w\sigma(\mathcal I)f(w) =0, 
	\end{align*}
	for all $\mathcal I\in \mathbb S$ and for any closed, homotopic to a point and  piecewise $C^1$ curve $\Gamma\subset U_{\mathcal I}$. 
	
	If $\{z \in U_{\mathcal I} \ \mid \ e^{\lambda\Re z} f (z) = 0\}$ has an accumulation point  in $U_{\mathcal I}$  then $f \equiv 0$  on $U$. On the other hand, if $\ell\in C(U, \mathbb R_n)$ satisfies   
	\begin{align*} 
		\int_{\Gamma} e^{\lambda \Re w} d_w\sigma(\mathcal I)  \ell (w) =0,
	\end{align*}
	for any closed, homotopic to a point and  piecewise $C^1$ curve $\Gamma\subset U_{\mathcal I}$ and for all $\mathcal I \in \mathbb S$, then $\ell\in \mathcal{SM}_{\lambda} (U)$.
\end{proposition}
\begin{proof}
	Corollary \ref{corSMLambda} shows that the mapping $u+\mathcal I v \mapsto  e^{u\lambda} f_{\mathcal I} (u + \mathcal I v)$, for all $u+\mathcal I v\in U$, where $u,v\in \mathbb R$ and $\mathcal I\in \mathbb S$, satisfies the first three facts of Proposition \ref{properties}. Finally the last property  follows from  Corollary \ref{corSMLambda} and the last fact of Proposition \ref{properties}. 
\end{proof}

On account of the Representation Theorem given in Proposition \ref{cor-Rep-Split-lambda} we can extend some formulas of Proposition \ref{propertieslambda} on all axially symmetric s-domain $U$. 

\section{Fractional slice monogenic functions with respect to a pair of real valued functions}\label{2}
\subsection{Riemann-Liouville fractional slice monogenic functions with respect to a pair of real valued functions}
Let  $a, b, c\in\mathbb R$ with $a<b$, $c>0$ and let us consider the special class of axially symmetric s-domains in $\mathbb R^{n+1}$ defined by
$$\mathcal S_{a , b,c }=\{    u+ {\mathcal  I} v  \in \mathbb R^{n+1}  \  \mid  \ u \in [ a, b] ,  \ v \in [0,c], \  {\mathcal  I} \in \mathbb S \}.$$
Given ${\mathcal  I}\in \mathbb S$, denote $\mathcal S_{a , b,c,{\mathcal  I} } := \mathcal  S_{a , b,c }\cap \mathbb C_{\mathcal  I} 
\subset \{u+{\mathcal  I}v; \ |\ v \geq 0\}$. Consider $r\in [a,b]$, $s\in [0,c]$ and $\alpha, \beta \in (0,1)$.

Let us denote by $AC^1(\mathcal S_{a , b,c }, \mathbb R_n)$ the class of $\mathbb R_n$-valued continuously functions $f$  on $\mathcal S_{a , b,c }$ such that  the mapings  $u\mapsto  f\mid_{\mathcal S_{a , b, c, {\mathcal I} }}(u+{\mathcal I}s)$ and $v \mapsto  f\mid_{\mathcal S _{a , b, c, {\mathcal I} }}(r+{\mathcal I}v)$ are absolutely continuous for all $u\in[a,b]$, $v\in [0,c]$. Given $g \in C^1([a,b], \mathbb R)$ and $h \in C^1([0,c], \mathbb R)$ such that $g'(t) >0$ for all  $t\in [a,b]$ and  $ h'(t)>0$ for all  $t\in [0,c]$.   

Write  
{\footnotesize
	\begin{align}\label{ecua1}
		[{\bf I}_{a^+}^{1-\alpha , g} f\mid_{\mathcal S_{a , b,c,{\mathcal  I} }}](u+{\mathcal I}s) := & \frac{1}{\Gamma(1-\alpha )}
		\int_a^u \frac{f\mid_{\mathcal S_{a , b,c,{\mathcal I} }}(\tau + {\mathcal I} s)}{(g(u)-g(\tau))^{\alpha } } g'(\tau) d\tau, \ \textrm{with} \ u > a, \nonumber \\
		[{\bf I}_{0^+}^{1-\beta , h} f\mid_{\mathcal S_{a , b, c, {\mathcal I} }}](r+{\mathcal I}v) :=& \frac{1}{\Gamma(1-\beta ) } \int_{0}^v 
		\frac{f\mid_{\mathcal S_{a , b,c, {\mathcal I} }}(r + {\mathcal I}\tau )}{(h(v)-h(\tau))^{ \beta }}  h'(\tau) d\tau, \ \textrm{with} \ v > 0 ,\nonumber \\
		[{\bf I}_{b^- }^{1-\alpha, g } f\mid_{\mathcal S_{a , b,c,{\mathcal I} }}](u+{\mathcal I}s) := & \frac{1}{\Gamma(1-\alpha )} \int_u^b 
		\frac{f\mid_{\mathcal S_{a , b,c,{\mathcal I} }}(\tau + {\mathcal I} s)}{( g(\tau) -g( u) )^{ \alpha}} g'(\tau) d\tau, \ \textrm{with} \ u < b, \nonumber  \\
		[{\bf I}_{c^-}^{1-\beta, h} f\mid_{\mathcal S_{a , b, c, {\mathcal I} }}](r+\mathcal I v) :=& \frac{1}{\Gamma(1-\beta )} \int_{v}^c \frac{f\mid_{
				\mathcal S_{a , b,c, {\mathcal I} }}(r + {\mathcal I}\tau )}{( h(\tau) - h( v))^{ \beta}}  h'(\tau) d\tau, \ \textrm{with} \ v <  c .
	\end{align}
}

Let us introduce  natural extensions of the real fractal derivatives \eqref{FracDer} and \eqref{FracDer1} preserving their structure.
\begin{definition}\label{FDQ} 
	The fractional derivatives on the left and on the right of $f \in AC^1(\mathcal S_{a , b,c }, \mathbb R_n)$, in the Riemann-Liouville sense,  associated with the slice $\mathbb C_{\mathcal  I}$ of order  $(\alpha,\beta)$ and with respect to $(g,h)$ are defined by 
	\begin{align*}
		& ({_{RL}}D _{a^+0^+,{\mathcal  I} }^{(\alpha,\beta),(g,h)} f\mid_{{ \mathcal  S_{a , b, c, {\mathcal  I} }}})(u+{\mathcal  I}v , r,s) := ({_{RL}}D _{a^+  }^{ \alpha,g }   f\mid_{\mathcal  S_{a , b,c,{\mathcal  I} }}) (u+{\mathcal  I}s) + {\mathcal  I} ({_{RL}}D _{0^+  }^{ \beta, h }f\mid_{\mathcal  S_{a , b,c,{\mathcal  I} }})(r+{\mathcal I}v)  \\
		& = \frac{1}{g'(u)} \frac{\partial }{\partial u}[{\bf I}_{a^+}^{1-\alpha , g} f\mid_{\mathcal S_{a , b,c,{\mathcal  I} }}](u+{\mathcal I}s)
		+\mathcal I
		\frac{1}{h'(v)} \frac{\partial }{\partial v} [{\bf I}_{0^+}^{1-\beta , h} f\mid_{\mathcal S_{a , b, c, {\mathcal I} }}](r+{\mathcal I}v),
		\\
		&({_{RL}}D _{b^-c^-, {\mathcal  I}}^{(\alpha,\beta),(g,h)} f\mid_{{\mathcal  S_{a , b, c, {\mathcal  I } }}})(u+\mathcal  I v, r,s):=  ({_{RL}}D _{b^-  }^{\alpha, g }   f\mid_{\mathcal   S_{a , b,c,{\mathcal  I} }}) (u+{\mathcal  I}s) +\mathcal  I  ({_{RL}}D 
		_{c^-  }^{\beta, h}   f\mid_{\mathcal    S_{a , b,c,\mathcal  I  }})(r+{\mathcal  I} v)\\
		& = \frac{1}{g'(u)} \frac{\partial }{\partial u}[{\bf I}_{b^-}^{1-\alpha , g} f\mid_{\mathcal S_{a , b,c,{\mathcal  I} }}](u+{\mathcal I}s)
		+\mathcal I
		\frac{1}{h'(v)} \frac{\partial }{\partial v} [{\bf I}_{c^-}^{1-\beta , h} f\mid_{\mathcal S_{a , b, c, {\mathcal I} }}](r+{\mathcal I}v),
	\end{align*}
	respectively,  where $u$ and $v$ are the variables of partial fractional derivation. Similarly,  reader has no difficulty in assigning a meaning to ${_{RL}}D _{a^+c^-,{\mathcal  I} }^{(\alpha,\beta),(g,h)}$ and ${_{RL}}D _{b^-0^+, {\mathcal  I}}^{(\alpha,\beta),(g,h)}$:
	\begin{align*}
		& ({_{RL}}D _{a^+c^-,{\mathcal  I}}^{(\alpha,\beta),(g,h)} f\mid_{{ \mathcal  S_{a , b, c, {\mathcal  I} }}})(u+{\mathcal  I}v , r,s) := ({_{RL}}D _{a^+  }^{ \alpha,g } f\mid_{\mathcal  S_{a , b,c,{\mathcal  I} }}) (u+{\mathcal  I}s) + {\mathcal  I} ({_{RL}}D _{c^-  }^{ \beta, h }f\mid_{\mathcal  S_{a , b,c,{\mathcal  I} }})(r+{\mathcal I}v) ,\\
		&({_{RL}}D _{b^- 0^+, {\mathcal  I}}^{(\alpha,\beta),(g,h)} f\mid_{{\mathcal  S_{a , b, c, {\mathcal  I } }}})(u+\mathcal  I v, r,s):=  ({_{RL}}D _{b^-  }^{\alpha, g }   f\mid_{\mathcal   S_{a , b,c,{\mathcal  I} }}) (u+{\mathcal  I}s) +\mathcal  I  ({_{RL}}D 
		_{0^+  }^{\beta, h}   f\mid_{\mathcal    S_{a , b,c,\mathcal  I  }})(r+{\mathcal  I} v),
	\end{align*}
	In particular, if $(u,v) =(r,s)$, then 
	\begin{align*}
		& ({_{RL}}D _{a^+0^+,{\mathcal  I} }^{(\alpha,\beta),(g,h)} f\mid_{{ \mathcal  S_{a , b, c, {\mathcal  I} }}})(r+{\mathcal  I}s , r,s)  = ({_{RL}}D _{a^+  }^{ \alpha,g }   f\mid_{\mathcal  S_{a , b,c,{\mathcal  I} }}) (r+{\mathcal  I}s) + {\mathcal  I} ({_{RL}}D _{0^+  }^{ \beta ,h}f\mid_{\mathcal  S_{a , b,c,{\mathcal  I} }})(r+{\mathcal I}s) ,\\
		&=:  ({_{RL}}D _{a^+0^+,{\mathcal  I} }^{(\alpha,\beta),(g,h)} f\mid_{{ \mathcal  S_{a , b, c, {\mathcal  I} }}})(r+{\mathcal  I}s ) \end{align*}
	and similar notation is followed for other operators. 
\end{definition}
In general, there exist two options: 

\noindent

a) If $\alpha,\beta\in\mathbb C$ with  $0< \Re \alpha, \Re \beta <1$ then our fractal derivatives give us functions with values in the complex Clifford algebra  $\mathbb R_n(\mathbb C) $ which consist of  $\sum_{A} e_Ax_A$, where   $x_A\in\mathbb  C$.  

\noindent

b) If $\alpha=(\alpha_1,\alpha_2)$ and  $\beta=(\beta_1, \beta_2)$ then use the mapping $\alpha=(\alpha_1,\alpha_2)  \mapsto \alpha_1 + \mathcal I\alpha_2  \in \mathbb C_{\mathcal I}$, $\beta=(\beta_1,\beta_2) \mapsto \beta_1 + \mathcal I\beta_2  \in \mathbb C_{\mathcal I}$ for each  $\mathcal I\in \mathbb S$ with  $0< \alpha_1, \beta_1 <1$.  

For abbreviation, we will restrict ourselves to $\alpha,\beta\in(0,1)$.  

These slice  monogenic  fractional  operators  are  $\mathbb R_n$ right-linear operators. In addition, the right versions of the previous operators are given by
\begin{align*} 
	& ({_{RL}}D _{a^+0^+,{\mathcal I},r }^{(\alpha,\beta),(g,h) } f\mid_{{\mathcal S_{a , b, c, {\mathcal I} }}})(u+{\mathcal I}v , r,s) :=
	({_{RL}}D _{a^+  }^{\alpha ,g}   f\mid_{\mathcal  S_{a , b,c,{\mathcal I} }}) (u+{\mathcal I}s) +  ({_{RL}}D _{0^+  }^{\beta,h }
	f\mid_{\mathcal S_{a , b,c,{\mathcal I} }})(r+{\mathcal I}v) {\mathcal I} ,\\
	&({_{RL}}D _{b^-c^-, {\mathcal I},r}^{(\alpha, \beta),(g,h)} f\mid_{{\mathcal S_{a , b, c, {\mathcal I} }}})(u+{\mathcal I}v, r,s):=
	({_{RL}}D _{b^-  }^{\alpha,g}   f\mid_{\mathcal S_{a , b,c,{\mathcal I} }}) (u+{\mathcal I}s) +  ({_{RL}}D _{c^-  }^{\beta,h}
	f\mid_{\mathcal S_{a , b,c,{\mathcal I} }})(r+{\mathcal I}v) {\mathcal I},
\end{align*}
where $u$ and $v$ are the variables of partial derivation  and    $ {_{RL}}D _{a^+c^-,{\mathcal  I} ,r}^{(\alpha,\beta),(g,h)}$ and ${_{RL}}D _{b^-0^+, {\mathcal  I},r }^{(\alpha,\beta),(g,h)}$ can be defined similarly.

Let $\lambda\in \mathbb R$ and let $\delta_1, \delta_2\in C^2(\mathbb R, \mathbb R)$ be two solutions of the following differential equation: 
$$y''+2\lambda y' =0.$$
So from now consider $g \in C^2([a,b], \mathbb R)$ and $h \in C^2([0,c], \mathbb R)$ given by $g(u)= \delta_1(u) $ and $g'(u) >0$ for all $u\in [a,b]$. Similarly,  $h(v)=\delta_2(v)$ and $h'(v)>0$ for all $v\in [0,c]$.

\begin{proposition}\label{FracProp1}
	Let $f\in AC^1(\mathcal S_{a , b,c }, \mathbb R_n)$. Then
	{\tiny
		\begin{align*}
			& 2\lambda  \left[  \frac{1}{ g'(u)  }   [{\bf I}_{a^+}^{1-\alpha ,g} f\mid_{\mathcal  S_{a , b,c,\mathcal  I }}](u+\mathcal  I s)  + 
			\mathcal I  \frac{1}{ h'(v)  }     [{\bf I}_{0^+}^{1-\beta, h} f\mid_{\mathcal  S_{a , b, c, \mathcal  I }}](r+\mathcal  I v) \right]  
			+ ({_{RL}}D _{a^+0^+,{\mathcal  I} }^{(\alpha,\beta),(g,h)} f\mid_{{\mathcal  S_{a , b, c, \mathcal  I }}})(u+\mathcal  I v , r,s)\\
			=&
			2\overline{\partial}_{{\mathcal  I}}   \left\{ \frac{1}{g'(u)}  [{\bf I}_{a^+}^{1-\alpha ,g} f\mid_{\mathcal  S_{a , b,c,\mathcal  I }}](u+\mathcal  I s) +
			\frac{1}{h'(v)} [{\bf I}_{0^+}^{1-\beta, h} f\mid_{\mathcal  S_{a , b, c, \mathcal  I }}](r+\mathcal  I v)\right\},\\
			&
			2\lambda  \left[  \frac{1}{ g'(u)  }   [{\bf I}_{b^-}^{1-\alpha ,g} f\mid_{\mathcal  S_{a , b,c,\mathcal  I }}](u+\mathcal  I s)  + 
			\mathcal I  \frac{1}{ h'(v)  }     [{\bf I}_{c^-}^{1-\beta, h} f\mid_{\mathcal  S_{a , b, c, \mathcal  I }}](r+\mathcal  I v) \right] +
			({_{RL}}D _{b^-c^-, {\mathcal  I}}^{(\alpha,\beta),(g,h)} f\mid_{{\mathcal  S_{a , b, c, \mathcal  I }}})(u+\mathcal  I v, r,s)\\
			=&
			- 2 \overline{\partial}_{{\mathcal  I }}  \left\{  \frac{1}{g'(u)}  [{\bf I}_{b^-}^{1-\alpha, g} f\mid_{\mathcal  S_{a , b,c,\mathcal  I  }}](u+\mathcal  I  s)   +
			\frac{1}{h'(v)}  [{\bf I}_{c^-}^{1-\beta,h} f\mid_{\mathcal  S_{a , b, c, \mathcal  I }}](s+\mathcal  I v)
			\right\}, \\
			& 2\lambda  \left[  \frac{1}{ g'(u)  }   [{\bf I}_{a^+}^{1-\alpha ,g} f\mid_{\mathcal  S_{a , b,c,\mathcal  I }}](u+\mathcal  I s)  + 
			\mathcal I  \frac{1}{ h'(v)  }     [{\bf I}_{c^-}^{1-\beta, h} f\mid_{\mathcal  S_{a , b, c, \mathcal  I }}](r+\mathcal  I v) \right]
			+ ({_{RL}}D _{a^+c^-,{\mathcal  I} }^{(\alpha,\beta),(g,h)} f\mid_{{\mathcal  S_{a , b, c, \mathcal  I }}})(u+\mathcal  I v , r,s)\\
			=&
			2\overline{\partial}_{{\mathcal  I}}   \left\{  \frac{1}{g'(u)}   [{\bf I}_{a^+}^{1-\alpha ,g} f\mid_{\mathcal  S_{a , b,c,\mathcal  I }}](u+\mathcal  I s) +
			\frac{1}{h'(v)}  [{\bf I}_{c^-}^{1-\beta, h} f\mid_{\mathcal  S_{a , b, c, \mathcal  I }}](r+\mathcal  I v)\right\},\\
			& 2\lambda  \left[  \frac{1}{ g'(u)  }   [{\bf I}_{b^-}^{1-\alpha ,g} f\mid_{\mathcal  S_{a , b,c,\mathcal  I }}](u+\mathcal  I s)  + 
			\mathcal I  \frac{1}{ h'(v)  }     [{\bf I}_{0^+}^{1-\beta, h} f\mid_{\mathcal  S_{a , b, c, \mathcal  I }}](r+\mathcal  I v) \right]+
			({_{RL}}D _{b^-0^+, {\mathcal  I}}^{(\alpha,\beta),(g,h)} f\mid_{{\mathcal  S_{a , b, c, \mathcal  I }}})(u+\mathcal  I v, r,s)\\
			=&
			- 2 \overline{\partial}_{{\mathcal  I }}  \left\{ \frac{1}{g'(u)}  [{\bf I}_{b^-}^{1-\alpha, g} f\mid_{\mathcal  S_{a , b,c,\mathcal  I  }}](u+\mathcal  I  s)   +
			\frac{1}{h'(v)}  [{\bf I}_{0^+}^{1-\beta,h} f\mid_{\mathcal  S_{a , b, c, \mathcal  I }}](s+\mathcal  I v)
			\right\},
		\end{align*}
	}
	for all $u+\mathcal I v \in \mathbb R^{n+1}$, where $u$ and $v$ are the variables of partial derivation of $\overline{\partial}_{{\mathcal I}} $.
\end{proposition}
\begin{proof}
	We give only the main ideas of the proof for the first identity, the other cases run along similar lines. 
	\begin{align*}
		&  2\overline{\partial}_{{\mathcal  I}}   \left\{ \frac{1}{g'(u)}  [{\bf I}_{a^+}^{1-\alpha ,g} f\mid_{\mathcal  S_{a , b,c,\mathcal  I }}](u+\mathcal  I s) +
		\frac{1}{h'(v)} [{\bf I}_{0^+}^{1-\beta, h} f\mid_{\mathcal  S_{a , b, c, \mathcal  I }}](r+\mathcal  I v)\right\} \\
		=&   \frac{\partial  }{\partial u} \left( \frac{1}{g'(u)}  [{\bf I}_{a^+}^{1-\alpha ,g} f\mid_{\mathcal  S_{a , b,c,\mathcal  I }}](u+\mathcal  I s) \right)
		+ \mathcal I  \frac{\partial  }{\partial u} \left(
		\frac{1}{h'(v)} [{\bf I}_{0^+}^{1-\beta, h} f\mid_{\mathcal  S_{a , b, c, \mathcal  I }}](r+\mathcal  I v) \right) \\
		=&  \frac{\partial  }{\partial u} \left( \frac{1}{g'(u)} \right)   [{\bf I}_{a^+}^{1-\alpha ,g} f\mid_{\mathcal  S_{a , b,c,\mathcal  I }}](u+\mathcal  I s)  + \mathcal I  \frac{\partial  }{\partial u} \left(
		\frac{1}{h'(v)} \right) [{\bf I}_{0^+}^{1-\beta, h} f\mid_{\mathcal  S_{a , b, c, \mathcal  I }}](r+\mathcal  I v)   \\
		&+  \frac{1}{g'(u)} \frac{\partial  }{\partial u}    [{\bf I}_{a^+}^{1-\alpha ,g} f\mid_{\mathcal  S_{a , b,c,\mathcal  I }}](u+\mathcal  I s)  
		+ \mathcal I \frac{1}{h'(v)}  \frac{\partial  }{\partial u}  
		[{\bf I}_{0^+}^{1-\beta, h} f\mid_{\mathcal  S_{a , b, c, \mathcal  I }}](r+\mathcal  I v)   \\
		= & 2\lambda  \left[  \frac{1}{ g'(u)  }   [{\bf I}_{a^+}^{1-\alpha ,g} f\mid_{\mathcal  S_{a , b,c,\mathcal  I }}](u+\mathcal  I s)  + 
		\mathcal I  \frac{1}{ h'(v)  }     [{\bf I}_{0^+}^{1-\beta, h} f\mid_{\mathcal  S_{a , b, c, \mathcal  I }}](r+\mathcal  I v) \right] \\
		& + ({_{RL}}D _{a^+0^+,{\mathcal  I} }^{(\alpha,\beta),(g,h)} f\mid_{{\mathcal  S_{a , b, c, \mathcal  I }}})(u+\mathcal  I v , r,s) ,\\
	\end{align*}
	since 
	$\displaystyle \frac{\partial}{\partial u} \left( \frac{1}{g'(u)} \right) = \frac{g''(u)}{(g'(u))^2} = 2\lambda \frac{1}{g'(u)}$ and $\displaystyle \frac{\partial  }{\partial v} \left( \frac{1}{h'(v)} \right) = \frac{h''(u)}{(h'(u) )^2} = 2\lambda \frac{1}{h'(u)}$. 
\end{proof}
Let us make the definition of Riemann-Liouville fractional slice monogenic functions, of order $(\alpha ,\beta)$ and with respect to $(g,h)$ on $\mathcal S_{a , b,c }$
\begin{definition}\label{Fract_s_m_function}
	Let $f\in AC^1(\mathcal S_{a,b,c}, \mathbb R_n)$ such that  
	\begin{align}\label{m1}
		&u\mapsto [{\bf I}_{a^+}^{1-\alpha,g } f\mid_{\mathcal S_{a , b,c, \mathcal I }}](u+ \mathcal I s), \quad  u\in [a,b], \nonumber \\
		&v\mapsto [{\bf I}_{0^+}^{1-\beta,h} f\mid_{\mathcal S_{a , b, c, \mathcal I  }}](r+ \mathcal I v ), \quad  v\in[0,c], 
	\end{align}
	are mappings of $C^1-$class. Then f will be called a Riemann-Liouville fractional slice monogenic functions, of order $(\alpha ,\beta)$ and with respect to $(g,h)$ on $\mathcal S_{a , b,c }$ if
	\begin{align*}
		& ({_{RL}}D _{a^+ 0^+,  \mathcal I }^{(\alpha,\beta),(g,h)} f\mid_{{\mathcal S_{a , b, c, \mathcal I}}})( u+ \mathcal  I v , r,s) =0 ,\quad  
		\forall  u+\mathcal I v \in  \mathcal  S_{a , b, c, \mathcal  I},
	\end{align*} 
	for all  $\mathcal I \in \mathbb S$. 
	
	The $\mathbb R_n$-right linear space of (left) Riemann-Liouville fractional slice monogenic functions, of order $(\alpha ,\beta)$ and with respect to $(g,h)$ on $\mathcal S_{a , b,c }$ is denoted by ${}_{{RL}}\mathcal {SM}^{(\alpha,\beta)}_{(g,h)}(\mathcal S_{a , b,c})$.
\end{definition}
The $\mathbb R_n$-left linear space of (right) Riemann-Liouville fractional slice monogenic functions, of order $(\alpha ,\beta)$ and  with respect to $(g,h)$ on $\mathcal S_{a , b,c }$ is denoted by ${}_{{RL}}\mathcal {SM}^{(\alpha,\beta)}_{(g,h), r} (\mathcal S_{a , b,c} )$ which consists of $f\in AC^1(\mathcal S_{a,b,c}, \mathbb R_n)$ satisfying \eqref{m1} and  
\begin{align*}
	& ({_{RL}}D _{a^+ 0^+,\mathcal I, r }^{(\alpha,\beta),(g,h)} f\mid_{{\mathcal S_{a , b, c, \mathcal I }}})(u+\mathcal I v , r,s) =0 ,\quad \forall u+\mathcal I v \in  \mathcal S_{a , b, c, \mathcal I},
\end{align*}
for all  $\mathcal I\in \mathbb S$.

\begin{remark}
	The kernels of the  operators ${_{RL}}D _{a^+c^-,{\mathcal I}}^{(\alpha,\beta),(g,h)}$, ${_{RL}}D _{b^-0^+, {\mathcal  I}}^{(\alpha,\beta),(g,h)}$, $ {_{RL}}D _{a^+c^-,{\mathcal  I} ,r}^{(\alpha,\beta),(g,h)}$ and 
	${_{RL}}D _{b^-0^+, {\mathcal  I},r }^{(\alpha,\beta),(g,h)}$ induce fractional slice monogenic function theories similar to ${}_{{RL}}\mathcal {SM}^{(\alpha,\beta)}_{(g,h)} (\mathcal S_{a , b,c} )$.
	
	On the other hand, the extension of \cite[Example 1]{GBS} to the fractional slice monogenic function theory can be used to give a non trivial element of $\mathcal {SM}^{(\alpha,\beta)}_{(g,h)} (\mathcal S_{a , b,c})$.
\end{remark}

\begin{remark}
	If $\lambda =0$ then functions $g$ and $h$ are given by $g(u)= \delta_1 u +\delta_2 $ for all $u\in[a,b ]$ and  $h(v)= \rho_1 v +\rho_2$ for all $v\in[0,c ]$,  where $\delta_k, \rho_k \in \mathbb R$ for  $k=1,2$ and $\delta_1, \rho_1> 0$. This case extends in the theory of slice monogenic functions the theory presented in \cite{GBS}.
	
	On the other hand, the case $\lambda\neq 0$ introduce a new fractional approach to a notion of slice monogenic functions with respect to the  pair $(g, h)$ given by $g(u)= \delta_1 e^{-2\lambda u} + \delta_2$ for all $u\in[a,b ]$ and  $h(v)= \rho_1 e^{-2\lambda v} + \rho_2$ for all $v\in[0,c ]$ where $\delta_k,\rho_k\in \mathbb R$ for $k=1,2$ with $-2\delta_1\lambda, -2\rho_1\lambda> 0$.
\end{remark}

\begin{proposition}\label{FracProp2}
	Suppose that $f\in AC^1(S_{a , b,c }, \mathbb H)$ satisfies (\ref{m1}). Then $f\in {}_{{RL}}\mathcal{SM}^{(\alpha,\beta)}_{(g,h)}(\mathcal S_{a , b, c })$ if and only if the mapping
	\begin{align*}
		q=u+\mathcal I  v \mapsto  \frac{1}{g'(u)}[{\bf I}_{a^+}^{1-\alpha,g} f\mid_{\mathcal S_{a , b,c,\mathcal I }}](u+\mathcal I s) +
		\frac{1}{h'(v)} [{\bf I}_{0^+}^{1-\beta,h} f\mid_{\mathcal S_{a , b, c, \mathcal I }}](r+\mathcal I v)
	\end{align*}
	belongs to $\mathcal{SM}_{\lambda}(\mathcal S_{a , b, c })$.
\end{proposition}
\begin{proof}
	Use Definition \ref{Fract_s_m_function} and Proposition \ref{FracProp1}.
\end{proof}

\begin{corollary} 
	Given $f\in AC^1(S_{a , b,c }, \mathbb H)$ satisfying (\ref{m1}) we see that $f \mid_{\mathcal S_{a , b, c , \mathcal I}} \in  \textrm{Ker} ({_{RL}}D _{a^+c^-,{\mathcal  I}}^{(\alpha,\beta),(g,h)})$ for all $\mathcal I\in \mathbb S^2$ if and only if 
	\begin{align*} 
		u+\mathcal I  v \mapsto  [{\bf I}_{a^+}^{1-\alpha,g} f\mid_{\mathcal S_{a , b,c,\mathcal I }}](u+\mathcal I s) +
		[{\bf I}_{c^-}^{1-\beta,h} f\mid_{\mathcal S_{a , b, c, \mathcal I }}](r+\mathcal I v), \quad \forall  u+\mathcal I  v\in \mathcal S_{a , b,c,\mathcal I}
	\end{align*}
	belongs to $ \mathcal{SM}_{\lambda}(\mathcal S_{a , b, c })$. In addition, $f \mid_{\mathcal S_{a , b, c , \mathcal I}} \in \textrm{Ker} ({_{RL}}D _{b^-0^+, {\mathcal  I} }^{(\alpha,\beta),(g,h)}  )$ for all $\mathcal I\in \mathbb S^2$ if and only if 
	\begin{align*} 
		u+\mathcal I  v \mapsto  [{\bf I}_{b^-}^{1-\alpha,g} f\mid_{\mathcal S_{a , b,c,\mathcal I }}](u+\mathcal I s) +
		[{\bf I}_{0^+}^{1-\beta,h} f\mid_{\mathcal S_{a , b, c, \mathcal I }}](r+\mathcal I v) , \quad \forall  u+\mathcal I  v\in \mathcal S_{a , b,c,\mathcal I }
	\end{align*}
	belongs to $ \mathcal{SM}_{\lambda}(\mathcal S_{a , b, c })$. Similarly, $f \mid_{\mathcal S_{a , b, c , \mathcal I}} \in  \textrm{Ker } ( {_{RL}}D _{b^-c^-, {\mathcal  I} }^{(\alpha,\beta),(g,h)}  )$ for all $\mathcal I\in \mathbb S^2$ iff 
	\begin{align*} 
		u+\mathcal I  v \mapsto  [{\bf I}_{b^-}^{1-\alpha,g} f\mid_{\mathcal S_{a , b,c,\mathcal I }}](u+\mathcal I s) +
		[{\bf I}_{c^-}^{1-\beta,h} f\mid_{\mathcal S_{a , b, c, \mathcal I }}](r+\mathcal I v) , \quad \forall  u+\mathcal I  v\in \mathcal S_{a , b,c,\mathcal I }
	\end{align*}
	belongs to $ \mathcal{SM}_{\lambda}(\mathcal S_{a , b, c })$. 
\end{corollary}
\begin{proof}
	Similar to the previous proof.
\end{proof}

\begin{proposition}(Representation Formula)  Let  $f\in {}_{{RL}}\mathcal{SM}_{(g,h)}^{(\alpha,\beta)}(\mathcal S_{a , b, c })$, then 
	for every $x=u+\mathcal I_x v \in \mathcal S_{a,b,c}$ and  $\mathcal I \in\mathbb S$ we have that 
	{\small
		\[\begin{split}
			& \frac{1}{g'(u)} [   {\bf I}_{a^+}^{1-\alpha ,g } f\mid_{\mathcal S_{a , b,c,\mathcal  I_x }}](u+\mathcal  I_x s )    +
			\frac{1}{h'(v)} [{\bf I}_{0^+}^{1-\beta, h } f\mid_{\mathcal  S_{a , b, c, \mathcal  I_x }}](r+\mathcal  I_x  v )  \\
			= & \frac{1}{2}\left\{  \frac{1}{g' } \left[  \  {\bf I}_{a^+}^{1-\alpha,g} [
			f\mid_{\mathcal  S_{a , b,c,\mathcal  I }}
			-   \mathcal  I_x \mathcal  I  f\mid_{\mathcal  S_{a , b,c,  \mathcal  I  }}  ]   +     {\bf I}_{a^+}^{1-\alpha,g } [
			f\mid_{\mathcal  S_{a , b,c,-\mathcal  I }}  +   \mathcal  I_x \mathcal I
			f\mid_{\mathcal  S_{a , b,c,- \mathcal  I }} ]   \   \right]  \right\} (u+\mathcal I s ) \\
			&  +
			\frac {1}{2} \left\{   \frac{1}{h' } \left[   \   {\bf I}_{0^+}^{1-\beta,h} [
			f\mid_{   \mathcal  S  _{  a , b,c,\mathcal  I }}
			-    \mathcal  I_x \mathcal  I
			f\mid_{\mathcal  S_{a , b,c,\mathcal I }}]    +
			{\bf I}_{0^+}^{1-\beta,h }[
			f\mid_{\mathcal  S_{a , b,c,-\mathcal I  }} +   {\mathcal I_x} {\mathcal I}
			f\mid_{S_{a , b,c,-{\mathcal I} }}   ] \right]  \   \right\} (r+{\mathcal I }v ) .
		\end{split} \] }
\end{proposition}

\begin{proof}
	Use the mappings 
	\begin{align*}
		u+{\mathcal I}  v \mapsto &   \frac{1}{g'(u)} [{\bf I}_{a^+}^{1-\alpha,g} f\mid_{\mathcal  S_{a , b,c, \mathcal  I }}](u+ \mathcal  I s ) +
		\frac{1}{h'(v)} [{\bf I}_{0^+}^{1-\beta ,h} f\mid_{\mathcal  S_{a , b, c,  \mathcal  I  }}](r+\mathcal  I v) , \\
		u-{\mathcal I}  v \mapsto &   \frac{1}{g'(u)} [{\bf I}_{a^+}^{1-\alpha,g} f\mid_{\mathcal  S_{a , b,c,-\mathcal  I }}](u+ \mathcal  I s ) +
		\frac{1}{h'(v)}  [{\bf I}_{0^+}^{1-\beta,h } f\mid_{\mathcal  S_{a , b, c, -\mathcal  I  }}](r+\mathcal  I v),  
	\end{align*}
	Proposition \ref{FracProp2} and   identity \eqref{represtheoremlambda} of Proposition \ref{cor-Rep-Split-lambda}. 
\end{proof}

\begin{proposition} (Splitting lemma) 
	Given   $f\in {}_{{RL}}\mathcal{SM}_{(g,h)}^{(\alpha,\beta)}(\mathcal S_{a , b, c}) $  and  $\mathcal I = \mathcal I_1 \in \mathbb  S$.  Let $\mathcal I_2,\dots , \mathcal I_n\in \mathbb S$ such that $\mathcal I_r \mathcal I_s + \mathcal I_s \mathcal I_r = - 2\delta_{rs}$.
	Then there exist $2^{n-1}$ many holomorphic functions $F_A : U_{\mathcal I}  \to \mathbb C_{\mathcal I}$ such that
	\begin{align*}
		\frac{1}{g'(\Re z)}[{\bf I}_{a^+}^{1-\alpha,g} f\mid_{\mathcal S_{a , b,c,\mathcal I }}](\Re z+\mathcal I s) +
		\frac{1}{h'(\Im z)}  [{\bf I}_{0^+}^{1-\beta,h} f\mid_{\mathcal S_{a , b, c, \mathcal I }}](r+\mathcal I \Im z)
		=\sum^{n-1}_{|A|=0} e^{-\lambda \Re z}F_A(z) {\mathcal I}_A, \end{align*}
	where $ {\mathcal I}_A={\mathcal I}_{i_1}\cdots{\mathcal I}_{i_s}$ and   $A=\{i_1\cdots  i_s\}$ is a subset of $\{2,\dots, n\}$, with $i_1<\dots< i_s$, or, $ I_{\emptyset}= 1$.
\end{proposition}
\begin{proof} 
	Use Proposition \ref{FracProp2} and  identity \eqref{serieslambda} of Proposition \ref{cor-Rep-Split-lambda}. 
\end{proof}

\begin{corollary}\label{Fract1.4lambda}  Some properties of $\mathcal{SM}_{(g,h)}^{(\alpha,\beta)}(\mathcal S_{a , b, c  }) $
	Let $f\in \mathcal{SM}_{(g,h)}^{(\alpha,\beta)}(\mathcal S_{a , b, c  })$. Suppose $u_0+ \mathcal I v_0 \in \mathcal S_{a , b, c}$
	and $r>0$ such that  $D_r:= \{ u+\mathcal I v \in \mathbb C_{\mathcal I} \ \mid \  (u-u_0)^2+ (v-v_0)^2\leq r^2 \} \subset \mathcal S_{a , b, c , \mathcal I }$.    
	
	\noindent
	i)  There exist a sequence $(C_n)_{n\geq 0}$ of elements of $\mathbb R_n$ such that 
	\begin{align}\label{fracsmseries}
		& \frac{1}{g'(\Re z)}[{\bf I}_{a^+}^{1-\alpha,g} f\mid_{\mathcal S_{a , b,c,\mathcal I }}](\Re z+\mathcal I s) +
		\frac{1}{h'(\Im z)}  [{\bf I}_{0^+}^{1-\beta,h} f\mid_{\mathcal S_{a , b, c, \mathcal I }}](r+\mathcal I \Im z)
		\nonumber  \\
		= & \sum_{n=0}^{\infty}  e^{-\lambda \Re z}(z- u_0-\mathcal I v_0)^n C_n, \end{align}
	for all $z\in \textrm{int}(D_r) $, the interior of $D_r$ contained in $\mathbb C_{\mathcal I}$.
	
	\noindent
	ii)  For any $z\in \textrm{int}(D_r)$ we see that 
	{\small
		\begin{align*}
			&\frac{1}{g'(\Re z)}  [{\bf I}_{a^+}^{1-\alpha,g} f\mid_{\mathcal S_{a , b,c,\mathcal I }}](\Re z+\mathcal I s) +
			\frac{1}{h'(\Im z)}   [{\bf I}_{0^+}^{1-\beta,h} f\mid_{\mathcal S_{a , b, c, \mathcal I }}](r+\mathcal I \Im z)
			\\ 
			= & \frac{1}{2\pi i}\int_{ \partial D_r } e^{ \lambda (\Re w- \Re z)} (w-z)^{-1} d_w\sigma(\mathcal I)  \left\{  \frac{1}{g'(\Re w)} [{\bf I}_{a^+}^{1-\alpha,g} f\mid_{\mathcal S_{a , b,c,\mathcal I }}](\Re w+\mathcal I s)  \right.  \\
			&  \left.  \  \  \   \   +
			\frac{1}{h'(\Im w)} [{\bf I}_{0^+}^{1-\beta,h} f\mid_{\mathcal S_{a , b, c, \mathcal I }}](r+\mathcal I \Im w) \right\},
	\end{align*}}
	where $\partial D_r$ the boundary of $D_r$ is contained in $\mathbb C_{\mathcal I}$. 
	
	\noindent
	iii) Let $ \Gamma $ be a closed piecewise $C^1-$curve, homotopic to a point, contained in $\textrm{int} (\mathcal S_{a , b, c, \mathcal I}).$ Then 
	\begin{align*} 
		& \int_{ \Gamma }   e^{\lambda \Re w} d_w\sigma(\mathcal I)
		\left\{ \frac{1}{g'(\Re w)}  [{\bf I}_{a^+}^{1-\alpha,g} f\mid_{\mathcal S_{a , b,c,\mathcal I }}](\Re w+\mathcal I s) +
		\frac{1}{h'(\Im w)}   [{\bf I}_{0^+}^{1-\beta,h} f\mid_{\mathcal S_{a , b, c, \mathcal I }}](r+\mathcal I \Im w) \right\} \\
		& =0. 
	\end{align*}
	
	\noindent
	iv)  If the set point 
	\begin{align*}
		\{ z\in \mathcal S_{a , b, c ,\mathcal I }  \ \mid \ h'(\Im z)   [{\bf I}_{a^+}^{1-\alpha,g} f\mid_{\mathcal S_{a , b,c,\mathcal I }}](\Re z+\mathcal I s) = -   
		g'(\Re z) [{\bf I}_{0^+}^{1-\beta,h} f\mid_{\mathcal S_{a , b, c, \mathcal I }}](r+\mathcal I \Im z) \}
	\end{align*}
	has an accumulation point in $\mathcal S_{a , b, c ,\mathcal I }$ then 
	\begin{align*}h'(\Im z)   [{\bf I}_{a^+}^{1-\alpha,g} f\mid_{\mathcal S_{a , b,c,\mathcal I }}](\Re z+\mathcal I s) = -   
		g'(\Re z) [{\bf I}_{0^+}^{1-\beta,h} f\mid_{\mathcal S_{a , b, c, \mathcal I }}](r+\mathcal I \Im z), \quad \forall z\in \mathcal S_{a , b, c ,\mathcal I }.
	\end{align*}
	
	\noindent
	v) Given  $\ell\in AC^1(  \mathcal S_{a , b, c  }, \mathbb R_n)$. Suppose that the mapping  
	\begin{align*}
		w\mapsto \frac{1}{g'(\Re w)}  [{\bf I}_{a^+}^{1-\alpha,g} \ell \mid_{\mathcal S_{a , b,c,\mathcal I }}](\Re w+\mathcal I s) +
		\frac{1}{h'(\Im w)}   [{\bf I}_{0^+}^{1-\beta,h} \ell \mid_{\mathcal S_{a , b, c, \mathcal I }}](r+\mathcal I \Im w), 
	\end{align*} 
	for all $w\in \mathcal S_{a , b, c ,\mathcal I }$, belongs to $C(\mathcal S_{a , b, c ,\mathcal I}, \mathbb R_n)$ and satisfies that 
	\begin{align*} \int_{\Gamma} d_w\sigma(\mathcal I)   e^{\lambda \Re w  } \left\{ 
		\frac{1}{g'(\Re w)}  [{\bf I}_{a^+}^{1-\alpha,g} \ell\mid_{\mathcal S_{a , b,c,\mathcal I }}](\Re w+\mathcal I s) +
		\frac{1}{h'(\Im w)}   [{\bf I}_{0^+}^{1-\beta,h} \ell\mid_{\mathcal S_{a , b, c, \mathcal I }}](r+\mathcal I \Im w)  \right\}
		=0,
	\end{align*}
	for any closed, homotopic to a point and  piecewise $C^1$ curve $\Gamma\subset \mathcal S_{a , b,c,\mathcal I}$ and for all ${\mathcal I}\in \mathbb S $. Then $\ell \in \mathcal{SM}_{(g,h)}^{(\alpha,\beta)}(\mathcal S_{a , b, c})$.
\end{corollary}

\begin{proof} 
	In  \eqref{serieslambda} of Proposition \ref{cor-Rep-Split-lambda} and in Proposition \ref{propertieslambda} use the mapping 
	\begin{align*}
		u+\mathcal I  v \mapsto  \frac{1}{g'(u)}[{\bf I}_{a^+}^{1-\alpha,g} f\mid_{\mathcal S_{a , b,c,\mathcal I }}](u+\mathcal I s) +
		\frac{1}{h'(v)} [{\bf I}_{0^+}^{1-\beta,h} f\mid_{\mathcal S_{a , b, c, \mathcal I }}](r+\mathcal I v)
	\end{align*}
	and  Proposition \ref{FracProp2}.
\end{proof}

\begin{remark}
	Given $\mathcal I\in \mathbb S$,  the previous corollary shows how  some properties of the slice monogenic functions are extending to mapgings given by    
	$$\frac{1}{g'(\Re z)}[{\bf I}_{a^+}^{1-\alpha,g} f\mid_{\mathcal S_{a , b,c,\mathcal I }}](\Re z+\mathcal I s) +
	\frac{1}{h'(\Im z)} [{\bf I}_{0^+}^{1-\beta,h} f\mid_{\mathcal S_{a , b, c, \mathcal I }}](r+\mathcal I \Im z)
	, \quad z\in  \mathcal S_{a , b, c, \mathcal I },$$
	for all $f\in {}_{{RL}}\mathcal{SM}^{\alpha,\beta}_{(g,h)}(\mathcal S_{a , b, c})$, based on well-known results of the theory of slice monogenic functions. Moreover, using Representation formula the previous formulas are extended to 
	$$\frac{1}{g'(u)}[{\bf I}_{a^+}^{1-\alpha, g} f](u+\mathcal I_x s) +\frac{1}{h'(v)} [{\bf I}_{0^+}^{1-\beta,h} f](r+\mathcal I_x v),$$
	for all  $u+\mathcal I_x v \in \mathcal S_{a , b, c,} $ and $f\in \mathcal{SM}^{(\alpha,\beta)}_{(g,h)}(\mathcal S_{a , b, c})$. 
	
	Taking into account the necessary hypothesis for the fundamental theorem of fractional calculus, applying the fractional derivatives ${D}_{a^+}^{1-\alpha,g}$ and ${D}_{0^+}^{1-\beta,h}$ on $h'(v)[{\bf I}_{a^+}^{1-\alpha, g} f](u+\mathcal I s) +g'(u) [{\bf I}_{0^+}^{1-\beta,h} f](r+\mathcal I v)$, we obtain formulas for $h'(v) f (u+\mathcal I s) + g'(u) f (r+\mathcal I v)$. For example, we will see the consequence of \eqref{fracsmseries} for $h'(v) f (u+\mathcal I s) + g'(u) f (r+\mathcal I v)$ repeating the previous ideas.     
	
	From \eqref{fracsmseries} and Representation Theorem we have that 
	\begin{align*}
		&\frac{1}{g'(u)}[{\bf I}_{a^+}^{1-\alpha,g} f\mid_{\mathcal S_{a , b,c,\mathcal I_x }}](u+\mathcal I_x s) +
		\frac{1}{h'(v)}  [{\bf I}_{0^+}^{1-\beta, h} f\mid_{\mathcal S_{a , b, c, \mathcal I_x }}](r+\mathcal I_x u)  \\
		= & \frac{1}{2} (1 - \mathcal I_x \mathcal I)
		\sum_{n=0}^{\infty}  ( u + \mathcal I v - u_0-\mathcal I v_0)^n C_n + \frac{1}{2} (1 + \mathcal I_x \mathcal I  )
		\sum_{n=0}^{\infty}  ( u - \mathcal I v - u_0-\mathcal I v_0)^n C_n,
	\end{align*} 
	where  $(C_n)_{n\geq 0}$ is  a sequence  of elements of $\mathbb R_n$,  for every $(u- u_0)^2 +  (v- v_0)^2 <r^2 $  and  $\mathcal I_x \in\mathbb S$. 
	
	Finally, writing the factor $g'(u) h'(v)$ and applying  the real fractional derivatives ${D}_{a^+}^{1-\alpha, g} \circ {D}_{0^+}^{1-\beta,h}$ on both sides, then the fundamental theorem of fractional calculus, \eqref{FundTheorem}, allows to see that  
	\begin{align*}
		& h'(v)  f(u+\mathcal  I_x s) +  g'(u) f(r+\mathcal  I_x  v ) \\
		= & \frac{1}{2} (1 - \mathcal I_x \mathcal I)
		{D}_{a^+}^{1-\alpha,g} \circ {D}_{0^+}^{1-\beta,h}\left[ \sum_{n=0}^{\infty} g'(u)h'(v) ( u + \mathcal I v - u_0-\mathcal I v_0)^nC_n \right]  \\
		&+ \frac{1}{2} (1 + \mathcal I_x \mathcal I)
		{D}_{a^+}^{1-\alpha, g} \circ {D}_{0^+}^{1-\beta,h}  \left[ \sum_{n=0}^{\infty}  g'(u)h'(v) ( u - \mathcal I v - u_0-\mathcal I v_0)^n C_n \right] \\
		& - {D}_{0^+}^{1-\beta,h} [h'(v)] [{\bf I}_{a^+}^{1-\alpha,g } f\mid_{\mathcal S_{a , b,c,\mathcal  I_x }}](u+\mathcal  I_x s) -
		{D}_{a^+}^{1-\alpha,g }  [g'(u)][{\bf I}_{0^+}^{1-\beta, h} f\mid_{\mathcal  S_{a , b, c, \mathcal  I_x }}](r+\mathcal  I_x v),
	\end{align*} 
	for all  $ (u- u_0)^2 +  (v- v_0)^2 <r^2 $ and   $\mathcal I_x \in\mathbb S$, where  $u$ is the real variable of the fractional derivative ${D}_{a^+}^{1-\alpha,g} $ and $v$ is the variable  of $ \circ {D}_{0^+}^{1-\beta,h}$.
	
	So the previous computations can be repeated for the missing  statements presented in Corollary \ref{Fract1.4lambda}  to obtain some formulas  for the mapping $(u+\mathcal  I_x v) \mapsto  h'(v)  f(u+\mathcal  I_x s) +  g'(u) f(r+\mathcal  I_x  v ) $ for all $ (u- u_0)^2 +  (v- v_0)^2 <r^2 $ and   $\mathcal I_x \in\mathbb S$.
\end{remark}

\subsection{Caputo fractional slice monogenic functions with respect to a pair of real valued functions}
With the notations introduced in the previous section, and partially relying on similar reasoning, in this section we study Caputo fractional slice monogencic functions with respect to a pair of real valued functions.

\begin{definition}\label{Fract_s_m_functionCaputo} Let $r \in [a,b] $ and $s\in[0,c]$. 
	The fractional derivatives in the Caputo sense of $f\in AC^1(S_{a , b,c }, \mathbb H)$, satisfying \eqref{m1}, on the left (on the right) with order induced by $(\alpha,\beta)$ with respect to $(g,h)$ and associated to the slice $\mathbb C_{\mathcal I}$ are      
	\begin{align*}
		& ({_{C}}D _{a^+0^+,{\mathcal  I} }^{(\alpha,\beta),(g,h)} f\mid_{{ \mathcal  S_{a , b, c, {\mathcal  I} }}})(u+{\mathcal  I}v , r,s) := ({_{C}}D _{a^+  }^{ \alpha,g }   f\mid_{\mathcal  S_{a , b,c,{\mathcal  I} }}) (u+{\mathcal  I}s) + {\mathcal  I} ({_{C}}D _{0^+  }^{ \beta, h }f\mid_{\mathcal  S_{a , b,c,{\mathcal  I} }})(r+{\mathcal I}v)  ,\\
		&({_{C}}D _{b^-c^-, {\mathcal  I}}^{(\alpha,\beta),(g,h)} f\mid_{{\mathcal  S_{a , b, c, {\mathcal  I } }}})(u+\mathcal  I v, r,s):=  ({_{C}}D _{b^-  }^{\alpha, g }   f\mid_{\mathcal   S_{a , b,c,{\mathcal  I} }}) (u+{\mathcal  I}s) +\mathcal  I  ({_{C}}D 
		_{c^-  }^{\beta, h}   f\mid_{\mathcal    S_{a , b,c,\mathcal  I  }})(r+{\mathcal  I} v), 
	\end{align*}
	respectively, where $u$ and $v$ are the variables of partial fractional derivation. The right versions of the previous operators are given by 
	\begin{align*}
		& ({_{C}}D _{a^+0^+,{\mathcal  I} ,r }^{(\alpha,\beta),(g,h)} f\mid_{{ \mathcal  S_{a , b, c, {\mathcal  I} }}})(u+{\mathcal  I}v , r,s) := ({_{C}}D _{a^+  }^{ \alpha,g }   f\mid_{\mathcal  S_{a , b,c,{\mathcal  I} }}) (u+{\mathcal  I}s) + ({_{C}}D _{0^+  }^{ \beta, h }f\mid_{\mathcal  S_{a , b,c,{\mathcal  I} }})(r+{\mathcal I}v)  {\mathcal  I} ,\\
		&({_{C}}D _{b^-c^-, {\mathcal  I}, r }^{(\alpha,\beta),(g,h)} f\mid_{{\mathcal  S_{a , b, c, {\mathcal  I } }}})(u+\mathcal  I v, r,s):=  ({_{C}}D _{b^-  }^{\alpha, g }   f\mid_{\mathcal   S_{a , b,c,{\mathcal  I} }}) (u+{\mathcal  I}s) + ({_{C}}D 
		_{c^-  }^{\beta, h}   f\mid_{\mathcal    S_{a , b,c,\mathcal  I  }})(r+{\mathcal  I} v)  \mathcal  I . 
	\end{align*}
	Operators  ${_{C}}D _{a^+c^-,{\mathcal  I} }^{(\alpha,\beta),(g,h)}$,    
	${_{C}}D _{b^-0^+, {\mathcal  I}}^{(\alpha,\beta),(g,h)}$,  ${_{C}}D _{a^+c^-,{\mathcal  I} ,r}^{(\alpha,\beta),(g,h)}$ and    
	${_{C}}D _{b^-0^+, {\mathcal  I}, r}^{(\alpha,\beta),(g,h)}$ are defined in similar way.
	
	Note that if $(u,v) =(r,s)$ then 
	\begin{align*}
		& ({_{C}}D _{a^+0^+,{\mathcal  I} }^{(\alpha,\beta),(g,h)} f\mid_{{ \mathcal  S_{a , b, c, {\mathcal  I} }}})(r+{\mathcal  I}s , r,s)  = ({_{C}}D _{a^+  }^{ \alpha,g }   f\mid_{\mathcal  S_{a , b,c,{\mathcal  I} }}) (r+{\mathcal  I}s) + {\mathcal  I} ({_{C}}D _{0^+  }^{ \beta ,h}f\mid_{\mathcal  S_{a , b,c,{\mathcal  I} }})(r+{\mathcal I}s) ,\\
		&=:  ({_{C}}D _{a^+0^+,{\mathcal  I} }^{(\alpha,\beta),(g,h)} f\mid_{{ \mathcal  S_{a , b, c, {\mathcal  I} }}})(r+{\mathcal  I}s ), \end{align*}
	and the missing operators follow the same notation. 
\end{definition}

\begin{remark}
	The previous operators are quaternionic extensions of \eqref{FracDerCaputo}. Even more, an interesting combination of fractional derivatives in the Riemann-Liouville and Caputo sense such that 
	\begin{align*}
		& ({_{RL}}D _{a^+  }^{ \alpha ,g}   f\mid_{\mathcal  S_{a , b,c,{\mathcal  I} }}) (u+{\mathcal  I}s) +   \mathcal  I 
		({_{C}}D _{0^+  }^{ \beta ,h}   f\mid_{\mathcal  S_{a , b,c,{\mathcal  I} }}) (r+{\mathcal  I}v) ,
		\\
		& ({_{RL}}D _{b^-  }^{ \alpha ,g}   f\mid_{\mathcal  S_{a , b,c,{\mathcal  I} }}) (u+{\mathcal  I}s) +   \mathcal  I 
		({_{C}}D _{0^+  }^{ \beta ,h}   f\mid_{\mathcal  S_{a , b,c,{\mathcal  I} }}) (r+{\mathcal  I}v) ,
		\\
		& 
		({_{C}}D _{b^-  }^{ \alpha ,g}   f\mid_{\mathcal  S_{a , b,c,{\mathcal  I} }}) (u+{\mathcal  I}s)  
		+{\mathcal  I} ({_{RL}}D _{c^-  }^{ \beta, h }  
		f\mid_{\mathcal  S_{a , b,c,{\mathcal  I} }})(r+{\mathcal I}v) ,  
	\end{align*}
	and some others similar operators can be considered. 
\end{remark}

\begin{definition}
	Let $f\in AC^1(\mathcal S_{a , b,c }, \mathbb R_n)$ satisfying \eqref{m1}. Then $f$ will be called a Caputo fractional slice monogenic   function of order $(\alpha, \beta)$, with respect to $(g,h)$ on $\mathcal S_{a , b,c }$ if
	\begin{align*}
		& ({}_{{ C}}D _{a^+,\mathcal I }^{(\alpha,\beta), (g,h)} f\mid_{{\mathcal  S_{a , b, c, \mathcal  I }}})(u+\mathcal  I v , r,s) =0 ,\quad  \forall u+\mathcal  I v \in \mathcal  S_{a , b, c, \mathcal I},
	\end{align*}
	for all  $\mathcal  I \in \mathbb S$. 
	
	The $\mathbb R_n$-right linear space of fractional monogenic functions, in Caputo sense is denoted by ${}_{{ C}}\mathcal{SM}^{(\alpha,\beta)}_{(g,h)}(\mathcal S_{a , b, c})$.
	
	On the other hand, by ${}_{{ C}}\mathcal{SM}^{(\alpha,\beta)}_{(g,h), r}(\mathcal S_{a , b, c})$ we denote the $\mathbb R_n$-left linear space of right-fractional slice monogenic  functions, in Caputo sense, on $\mathcal S_{a , b,c }$ and it consists of $f\in AC^1(\mathcal S_{a , b,c }, \mathbb R_n)$ satisfying \eqref{m1} and 
	\begin{align*}
		& ({_{{ C}}}D _{a^+,\mathcal I,r }^{(\alpha,\beta),(g,h)} f\mid_{{\mathcal S_{a , b, c, \mathcal I }}})(u+\mathcal I v , r,s) =0 ,\quad  \forall u+\mathcal I v \in {\mathcal S_{a , b, c, \mathcal I}},
	\end{align*}
	for all $\mathcal I \in \mathbb S.$ 
\end{definition}

\begin{proposition} 
	Suppose that $f\in AC^1(\mathcal S_{a , b,c }, \mathbb R_n)$ satisfies \eqref{m1}. Define the operator 
	\begin{align*}
		\mathcal H(f) (u+\mathcal I v, r,s):= ({\bf I}_{a^+}^{1-\alpha, g} f)(u+\mathcal I s) + ({\bf I}_{0^+}^{1-\beta, h} f)(r+ \mathcal I v). 
	\end{align*}
	Then 
	\begin{align*}
		({}_C D _{a^+, 0^+}^{(\alpha, \beta), (g,h )} \mathcal H(f)  (u+\mathcal I v, r,s)  = & \mathcal H\left[  ({}_{RL} D _{a^+, 0^+}^{(\alpha, \beta), (g,h )} \mathcal H(f)  (u+\mathcal I v, r,s)\right] \\
		& - {\bf I}_{0^+}^{1-\beta, h}  [1]  ({}_{RL}D _{a^+}^{\alpha, g} f)(u+\mathcal I s) - {\bf I}_{a^+}^{1-\alpha, g} [1] \mathcal I ({}_{RL}D _{0^+}^{\beta, h } f)(r+ \mathcal I v), 
	\end{align*}
	for all $u+\mathcal I v\in \mathcal S_{a , b,c }$ and $\mathcal I\in \mathbb S$. 
\end{proposition}
\begin{proof} From direct computation and making use of the identities \eqref{R_LandC1} and \eqref{R_LandC2} we have that
	\begin{align*}
		& ({}_C D _{a^+, 0^+}^{(\alpha, \beta), (g,h )} \mathcal H(f)  (u+\mathcal I v, r,s)\\ = & \left[{\bf I}_{a^+}^{1-\alpha, g}     \frac{1 }{ g' }  \frac{\partial }{\partial u} \mathcal H(f) (u+\mathcal I v, r,s)  \right] + \mathcal I  \left[{\bf I}_{0^+}^{1-\beta, h}     \frac{1 }{h'} \frac{\partial}{\partial v} \mathcal H(f) (u+\mathcal I v, r,s)  \right] \\
		= & \left[{\bf I}_{a^+}^{1-\alpha, g} \frac{1 }{g'} \frac{\partial}{\partial u}    
		({\bf I}_{a^+}^{1-\alpha, g} f) (u+\mathcal I s) \right] + \mathcal I  \left[{\bf I}_{0^+}^{1-\beta, h}\frac{1}{h'}\frac{\partial }{  \partial v } ({\bf I}_{0^+}^{1-\beta, h} f) (r+ \mathcal I v) \right] \\
		= & {\bf I}_{a^+}^{1-\alpha, g} ({}_{RL}D _{a^+}^{\alpha, g} f) (u+\mathcal I s) + \mathcal I {\bf I}_{0^+}^{1-\beta, h}({}_{RL}D _{0^+}^{\beta, h}f) (r+ \mathcal I v), 
	\end{align*}
	for all $u+\mathcal I v\in \mathcal S_{a , b,c }$ and $\mathcal I\in \mathbb S$. Therefore, 
	\begin{align*}
		({}_C D _{a^+, 0^+}^{(\alpha, \beta), (g,h )} \mathcal H(f)  (u+\mathcal I v, r,s)  = & \mathcal H \left[
		({}_{RL}D _{a^+}^{\alpha, g } f)  (u+\mathcal I s) + \mathcal I ({}_{RL}D _{0^+}^{\beta, h } f) (r+ \mathcal I v) \right] \\
		& - {\bf I}_{0^+}^{1-\beta, h} [1] ({}_{RL}D _{a^+}^{\alpha, g } f)  (u+\mathcal I s) - {\bf I}_{a^+}^{1-\alpha, g} [1] \mathcal I({}_{RL}D _{0^+}^{\beta, h } f)(r+\mathcal I v), 
	\end{align*}
	for all $u+\mathcal I v\in \mathcal S_{a , b,c }$ and $\mathcal I\in \mathbb S$. 
\end{proof}

So we obtain a characterization of the elements of ${}_{{ C}}\mathcal{SM}^{(\alpha,\beta)}_{(g,h)}(\mathcal S_{a , b, c})$,
in terms of the slice monogenic operator in the Riemann-Loiuville sense as follows:

\begin{corollary} 
	$f\in {}_{{ C}}\mathcal{SM}^{(\alpha,\beta)}_{(g,h)}(\mathcal S_{a , b, c})$ if and only if 
	\begin{align}\label{last}
		\mathcal H\left[({}_{RL} D_{a^+, 0^+}^{(\alpha, \beta), (g,h)} \mathcal H(f)(u+\mathcal I v, r,s)\right] = {\bf I}_{0^+}^{1-\beta, h}[1]({}_{RL}D_{a^+}^{\alpha, g}f) (u+\mathcal I s) + {\bf I}_{a^+}^{1-\alpha, g}[1] \mathcal I ({}_{RL}D _{0^+}^{\beta, h}f)(r+ \mathcal I v), 
	\end{align}
	for all $u+\mathcal I v\in \mathcal S_{a , b,c }$ and $\mathcal I\in \mathbb S$. 
	
	On the other hand, if $f\in {}_{{ C}}\mathcal{SM}^{(\alpha,\beta)}_{(g,h)}(\mathcal S_{a , b, c})$ we have  
	\begin{align*}
		& {}_{RL}{D}_{a^+}^{1-\alpha, g} \circ {}_{RL}{D}_{0^+}^{1-\beta, h} \left\{ \mathcal H\left[({}_{RL} D_{a^+, 0^+}^{(\alpha, \beta), (g,h )} \mathcal H(f)  (u+\mathcal I v, r,s)\right] \right\} \\
		= & {}_{RL}{D}_{a^+}^{1-\alpha, g} ({}_{RL}D_{a^+}^{\alpha, g} f)(u+\mathcal I s) + \mathcal I {}_{RL}{D}_{0^+}^{1-\beta, h} ({}_{RL}D_{0^+}^{\beta, h}f)(r+ \mathcal I v), 
	\end{align*}
	for all $u+\mathcal I v\in \mathcal S_{a , b,c}$ and $\mathcal I\in \mathbb S$. 
\end{corollary}
\begin{proof} 
	The firs fact follows from the previous proposition. If we apply ${D}_{a^+}^{1-\alpha, g} \circ {D}_{0^+}^{1-\beta, h} $ on both sides  of \eqref{last} the second fact follows.
\end{proof}

\section*{Statements and Declarations}
\subsection*{Funding} This work was partially supported by Instituto Polit\'ecnico Nacional (grant numbers SIP20241638, SIP20241237)).
\subsection*{Competing Interests} The authors declare that they have no competing interests regarding the publication of this paper.
\subsection*{Ethics approval and consent to participate} Not applicable.
\subsection*{Consent for publication} Not applicable.
\subsection*{Data availability} Not applicable.
\subsection*{Materials availability} Not applicable.
\subsection*{Code availability} Not applicable.
\subsection*{Author contribution} All authors contributed equally to the study, read and approved the final version of the submitted manuscript.

\subsection*{ORCID}
\noindent
Jos\'e Oscar Gonz\'alez-Cervantes: https://orcid.org/0000-0003-4835-5436\\
Juan Bory-Reyes: https://orcid.org/0000-0002-7004-1794




\begin{thebibliography}{999}
	\bibitem{ACDS} Alpay, D., Colombo, F., Diki, K., Sabadini, I. Poly slice monogenic functions, Cauchy formulas and the $PS$-functional calculus. {\em J. Operator Theory}, 88 {\bf 2022}, no. 2, 309-364. 
	\bibitem{BDS} Brackx, F.,  Delanghe, R., Sommen, F. Clifford Analysis, (Research Notes in Mathematics, Vol. 76), Boston, London, Melbourne: Pitman Advanced Publishing Company, {\bf 1982}. 
	\bibitem{CGS} Colombo, F., Gonz\'alez-Cervantes, J. O., Sabadini, I. Comparison of the various notions of slice monogenic functions and their variations. {\em In AIP Conference Proceedings}, Vol. 1389, No. 1, {\bf 2011}, 264-267. American Institute of Physics.
	\bibitem{CKPS} Colombo, F., Kimsey, D. P., Pinton, S., Sabadini, I. Slice monogenic functions of a Clifford variable via the $S$-functional calculus. {\em Proc. Amer. Math. Soc. Ser. B}, 8 {\bf 2021}, 281--296. 
	\bibitem{CSS1} Colombo, F., Sabadini, I., Struppa, D.C.  Slice monogenic functions, {\em Israel J. Math.}, 171, {\bf 2009},
	385--403.
	\bibitem{CSS2} Colombo, F.,  Sabadini, I.,  Struppa, D.C. An extension theorem for slice monogenic functions and some of its consequences, {\em Israel J. Math.}, 177, {\bf 2010},  369--389.
	\bibitem{CSS3} Colombo, F., Sabadini, I., Struppa, D.C. Slice monogenic functions. {\em In: Noncommutative Functional Calculus. Progress in Mathematics}, vol 289. Springer, Basel {\bf 2011}.
	\bibitem{CS1} Colombo,  F.,  Sabadini, I. A structure formula for slice monogenic functions and some of its consequences, Hypercomplex Analysis, {\em Trends in Mathematics, Birkhauser}, {\bf 2009}, 101--114. 
	\bibitem{CS2} Colombo, f.,  Sabadini, I. The Cauchy formula with s-monogenic kernel and a functional calculus for noncommuting operators, {\em J. Math. Anal. Appl.}, 373, {\bf 2011}, 655--679.
	\bibitem{CSSS} Colombo, F., Sabadini, I., Sommen, F.,  Struppa, D. C. Analysis of Dirac systems and computational algebra, {\em Progress in Mathematical Physics}, Vol. 39, Birkhauser Boston, {\bf 2004}.
	\bibitem{CDR} Cnudde, L., De Bie, H., Ren, G.. Algebraic approach to slice monogenic functions. {\em Complex Anal. Oper. Theory} 9 {\bf 2015}, no. 5, 1065--1087.
	\bibitem{CTOP} Coloma, N.,  Di Teodoro, A.,  Ochoa-Tocachi, D.,  Ponce, F. Fractional Elementary Bicomplex Functions in the Riemann-Liouville Sense. {\em Adv. Appl. Clifford Alg.} 31,  4. {\bf 2021}, 63--29. 
	\bibitem{DM} Delgado, B. B.,  Mac{\'i}as-D{\'i}az, J.E.  On the General Solutions of Some Non-Homogeneous Div-Curl Systems with Riemann-Liouville and Caputo Fractional Derivatives. {\em Fractal Fract}, 3, 3, {\bf 2021},  1045--1100. 
	\bibitem{DSS} Delanghe, R., Somman, F., Soucek, V. Clifford Algebra and Spinor-Valued Functions, Dorderecht, Boston, London: Kluwer Academic Publishers, {\bf 1992}.
	\bibitem{GM}  Gilbert, J. E.,  Murray, M. A. M. Clifford algebras and Dirac operators in harmonic analysis, Cambridge Studies in Advanced Mathematics, Vol. 26, Cambridge University Press, Cambridge, 1991.
	\bibitem{G}   Gonz\'alez-Cervantes, J. O. On some quaternionic generalized slice regular functions.  {\em Advances in Applied Clifford Algebras} 32 (2022), no. 3, Paper No. 36, 17 pp.
	\bibitem{GB1} Gonz\'alez-Cervantes, J. O., Bory-Reyes. J.  A quaternionic fractional Borel-Pompeiu type formula. {\em Fractal}, 30, No. 1 {\bf 2022} 2250013 (15 pages).
	\bibitem{GB2}  Gonz{\'a}lez-Cervantes J. O., Bory-Reyes, J. A bicomplex $(\vartheta,\varphi)-$weighted fractional Borel-Pompeiu type formula. {\em J. Math. Anal. Appl.}, 520, No. 2,  {\bf  2023}, 2025-2034. 
	\bibitem{GB3}  Gonz\'alez-Cervantes, J. O., Bory-Reyes. J.  A fractional Borel-Pompeiu type formula and a related fractional $\psi$-Fueter operator with respect to a vector-valued function. {\em Math. Methods Appl. Sci.} 46 {\bf 2023}, No. 2, 2012--2022.
	\bibitem{GBS} Gonz\'alez-Cervantes, J. O., Bory-Reyes,  J.,  Sabadini, I. Fractional Slice Regular Functions of a Quaternionic Variable. {\em Results Math} 79, 32, {\bf 2024}. 
	\bibitem{GHS} Gurlebeck, K., Habetha, K., Spro{\ss}ig, W. Holomorphic functions in the plane and $n-$dimensional space, Birkhauser Verlag, Basel, {\bf 2008}.
	\bibitem{KV} K{\"a}hler, U.  Vieira, N. Fractional Clifford analysis, {\em Hypercomplex analysis: new perspectives and applications. Trends in mathematics}, 520, 2,  {\bf 2014}. 191-201. 
	\bibitem{KST} Kilbas, A. A., Srivastava, H. M.,  Trujillo, J. J. Theory and Applications of Fractional Differential Equations.
	Hypercomplex analysis: new perspectives and applications.  Trends in mathematics. Amsterdam. North-Holland Mathematics Studies, 204. Elsevier Science B.V. {\bf 2006}.  
	\bibitem{JAA} Jarad, F., Alqudah, M.A., Abdeljawad, T. On more generalized form of proportional fractional operators. {\em Open Math},
	18, 167--176 (2020)
\bibitem{JARH}  Jarad, F.  Abdeljawad, T.,   Rashid, S.,  and Hammouch Z.  More properties of the proportional
fractional integrals and derivatives of a function with respect to another function. {\em Adv. Difference Equ.} 2020, Paper No. 303, 16 pp.
	\bibitem{MR} Miller, K. S.,  Ross, B.  An Introduction to the Fractional Calculus and Fractional Differential Equations. New York. A Wiley Interscience Publication. John Wiley \& Sons, Inc. {\bf 1993}.
	\bibitem{OS} Oldham, K. B. and Spanier, J. The Fractional Calculus, New York, Dover Publ. Inc. {\bf 2006}.
	\bibitem{O} Ortigueira, M. D. Fractional calculus for scientists and engineers. Lecture Notes in Electrical Engineering, 84, Dordrecht. Springer. {\bf 2011}.
	\bibitem{PBBB} Pe\~na P{\'e}rez, Y.,  Abreu Blaya, R.,  {\'A}rciga Alejandre, M. P., Bory Reyes, J. Biquaternionic reformulation of a fractional monochromatic Maxwell system. {\em Adv. High Energy Phys}. 13, 2. {\bf 2020}, 71--162.
	\bibitem{P} Podlubny, I. Fractional differential equations. An introduction to fractional derivatives, fractional differential equations, to methods of their solution and some of their applications. San Diego, CA. Mathematics in Science and Engineering, 198. Academic Press, Inc. {\bf 1999}.
	\bibitem{SKM} Samko, S.G., Kilbas, A. A., Marichev, O.I.  Fractional Integrals and Derivatives. Theory and Applications. New York. Gordon and Breach Sci. Publ. London. {\bf 1993}.
	\bibitem{V} Vieira, N. Fischer decomposition and Cauchy-Kovalevskaya extension in fractional Clifford analysis: the Riemann-Liouville case, {\em Proc. Edinb. Math. Soc. II.} 60, 1, {\bf  2017},  251--272.
	\bibitem{YQ} Yang, Y., Qian, T. Zeroes of slice monogenic functions. {\em Math. Methods Appl. Sci.} 34 {\bf  2011}, no. 11, 1398-1405.  
	\bibitem{XS} Xu, Z., Sabadini, I. Generalized partial-slice monogenic functions: a synthesis of two function theories. {\em Adv. Appl. Clifford Algebr.} 34 {\bf 2024}, no. 2, Paper No. 10, 14 pp. 	
\end{thebibliography}
\end{document}